\documentclass{amsart}
\usepackage{amstext,amssymb,amsthm,amsopn,newlfont,graphpap,graphics,graphicx,mathrsfs,enumitem}\usepackage{amstext,amssymb,amsthm,amsopn,newlfont,graphpap,graphics,graphicx,mathrsfs,mathtools,enumitem}
\allowdisplaybreaks
\usepackage[parfill]{parskip}
\usepackage[noadjust]{cite}
\usepackage{epigraph}
\usepackage[colorlinks=true,
            linkcolor=red,
            urlcolor=blue,
            citecolor=magenta]{hyperref}
\usepackage{color}
\usepackage{mathrsfs}
\allowdisplaybreaks
\theoremstyle{plain}
\newtheorem{thm}{Theorem}[section]
\newtheorem*{thm*}{Theorem}
\newtheorem{prop}{Proposition}[section]
\newtheorem*{prop*}{Proposition}
\newtheorem{cor}{Corollary}[section]
\newtheorem*{cor*}{Corollary}

\newtheorem*{lem*}{Lemma}
\theoremstyle{definition}
\newtheorem{defn}{Definition}[section]
\newtheorem*{defn*}{Definition}

\newtheorem*{exmps*}{Examples}

\newtheorem*{exmp*}{Example}

\newtheorem*{exerc*}{Exercise}

\newtheorem{rems}{Remarks}[section]
\newtheorem*{rems*}{Remarks}
\newtheorem{rem}[rems]{Remark}
\newtheorem*{rem*}{Remark}
\newcommand{\N}{{\mathbb N}}
\newcommand{\Z}{{\mathbb Z}}
\newcommand{\R}{{\mathbb R}}
\newcommand{\C}{{\mathbb C}}
\newcommand{\F}{{\mathbb F}}

\newcommand{\eps}{\varepsilon}

\numberwithin{equation}{section}

\begin{document}
\title[On convergence in Banach spaces with a Schauder basis]
{On a characterization of convergence\\ in Banach spaces with a Schauder basis}
\author[Marat V. Markin]{Marat V. Markin}
\address{
Department of Mathematics\newline
California State University, Fresno\newline
5245 N. Backer Avenue, M/S PB 108\newline
Fresno, CA 93740-8001
}
\email[corresponding author]{mmarkin@csufresno.edu}
\author{Olivia B. Soghomonian}
\email{osogho5780@mail.fresnostate.edu}
\begin{abstract}
We extend the well-known characterizations of convergence in the spaces $l_p$ ($1\le p<\infty$) of $p$-summable sequence and $c_0$ of vanishing sequences to a general characterization of convergence in a Banach space with a Schauder basis and obtain as instant corollaries characterizations of convergence in an infinite-dimensional separable Hilbert space and the space $c$ of convergent sequences.
\end{abstract}
\maketitle
\epigraph{\textit{The method in the present paper is abstract and is phrased in terms of Banach spaces, linear operators, and so on. This has the advantage of greater simplicity in proof and greater generality in applications.}}{Jacob T. Schwartz}

\section[Introduction]{Introduction}

In normed vector spaces of sequences, termwise convergence, being a necessary condition for convergence of a sequence (of sequences), falls short of being characteristic (see, e.g., \cite{Markin2020EOT}). Thus, the natural question: what conditions are required to be, along with termwise convergence, necessary and sufficient for convergence of a sequence in such spaces?

It turns out that, in the Banach spaces $l_p$ ($1\le p<\infty$) of $p$-\textit{summable sequences} with $p$-norm
\[
x:=\left(x_k\right)_{k\in \N}\mapsto \|x\|_p:={\left[\sum_{k=1}^\infty \left|x_k\right|^p\right]}^{1/p}
\]
 ($\N:=\left\{ 1,2,\dots\right\}$ is the set of \textit{natural numbers})
and $c_0$ of \textit{vanishing sequences} with $\infty$-norm
\begin{equation}\label{infn}
x:=\left(x_k\right)_{k\in \N}\mapsto \|x\|_\infty:=\sup_{k\in \N}|x_k|,
\end{equation} 
only one additional condition is needed. The following characterizations of convergence in the foregoing spaces are well-known.

\begin{prop}[Characterization of Convergence in $l_p$ ($1\le p<\infty$)]\label{CCl_p}\ \\
In the (real or complex) space $l_p$ ($1\le p<\infty$), 
\[
\left(x_k^{(n)}\right)_{k\in \N}=:x^{(n)}\to x:=\left(x_k\right)_{k\in \N},\ n\to \infty,
\]
iff
\begin{enumerate}[label={(\arabic*)}]
\item $\forall\, k\in\N:\ x_k^{(n)}\to x_k$, $n\to \infty$, and
\item $\displaystyle \forall\, \varepsilon>0\ \exists\, K\in \N\ \forall\,n\in \N:\ \sum_{k=K+1}^\infty \left|x_k^{(n)}\right|^p<\varepsilon$.
\end{enumerate}
\end{prop}

See, e.g., \cite[Proposition $2.16$]{Markin2018EFA}, \cite[Proposition $2.17$]{Markin2020EOT}.

\begin{rems}\
\begin{itemize}
\item Condition (1) is termwise convergence.
\item Condition (2) signifies the uniform convergence of the series
\[
\sum_{k=1}^\infty \left|x_k^{(n)}\right|^p
\]
to their respective sums over $n\in \N$.
\end{itemize}
\end{rems}

\begin{prop}[Characterization of Convergence in $c_0$]\label{CCc_0}\ \\
In the (real or complex) space $c_0$, 
\[
\left(x_k^{(n)}\right)_{k\in \N}=:x^{(n)}\to x:=\left(x_k\right)_{k\in \N},\ n\to \infty,
\]
iff
\begin{enumerate}[label={(\arabic*)}]
\item $\forall\, k\in\N:\ x_k^{(n)}\to x_k$, $n\to \infty$, and
\item $\displaystyle \forall\, \varepsilon>0\ \exists\, K\in \N\ \forall\,n\in \N:\ \sup_{k\ge K+1} \left|x_k^{(n)}\right|<\varepsilon$.
\end{enumerate}
\end{prop}

See, e.g., \cite[Proposition $2.15$]{Markin2018EFA}, \cite[Proposition $2.16$]{Markin2020EOT}.

\begin{rems}\
\begin{itemize}
\item Condition (1) is termwise convergence.
\item Condition (2) signifies the uniform convergence of the sequences $\left(x_k^{(n)}\right)_{k\in \N}$ to $0$ over $n\in \N$.
\end{itemize}
\end{rems}

One cannot but notice that both characterizations share the same condition  (1) and that condition (2) in each can be reformulated in the following equivalent form: 
\begin{enumerate}
\item[(2C)] $\displaystyle \forall\, \varepsilon>0\ \exists\, K_0\in \N\ \forall\,K\ge K_0\ \forall\,n\in \N:\ \left\|R_Kx^{(n)}\right\|<\eps$,
\end{enumerate}
where $\|\cdot\|$ stands for $p$-norm $\|\cdot\|_p$ ($1\le p<\infty$) or $\infty$-norm, respectively, and the mapping $R_K:X\to X$, $K\in \N$, ($X:=l_p$ ($1\le p<\infty$) or $X:=c_0$) is defined as follows:
\begin{equation}\label{remop}
x:=\left(x_k\right)_{k\in \N}\mapsto R_Kx:=(\underbrace{0,\dots,0}_{\text{$K$ terms}},x_{K+1},x_{K+2},\dots),\ K\in \N.
\end{equation}

Thus, we have the following combined characterization encompassing both $l_p$ ($1\le p<\infty$) and $c_0$.

\begin{prop}[Combined Characterization of Convergence]\label{CCC}\ \\
In the (real or complex) space $X:=l_p$ ($1\le p<\infty$) or $X:=c_0$, 
\[
\left(x_k^{(n)}\right)_{k\in \N}=:x^{(n)}\to x:=\left(x_k\right)_{k\in \N},\ n\to \infty,
\]
iff
\begin{enumerate}[label={(\arabic*)}]
\item $\forall\, k\in\N:\ x_k^{(n)}\to x_k$, $n\to \infty$, and
\item[(2C)] $\displaystyle \forall\, \varepsilon>0\ \exists\, K_0\in \N\ \forall\,K\ge K_0\ \forall\,n\in \N:\ \left\|R_Kx^{(n)}\right\|<\eps$,
\end{enumerate}
where $\|\cdot\|$ stands for $p$-norm $\|\cdot\|_p$ ($1\le p<\infty$) or $\infty$-norm, respectively, and the mapping $R_K:X\to X$, $K\in \N$, is defined by \eqref{remop}.
\end{prop}

In view of the fact that both $l_p$ ($1\le p<\infty$) and $c_0$ are Banach spaces with a Schauder basis, our goal to show that a two-condition characterization of convergence, similar to the foregoing combined characterization, holds for all such spaces appears to be amply motivated. We establish a general characterization of convergence in a Banach space with a Schauder basis and obtain as instant corollaries characterizations of convergence in an an infinite-dimensional separable Hilbert space and the Banach space $c$ of \textit{convergent sequences}.

\section[Preliminaries]{Preliminaries}

Here, we briefly outline certain preliminaries essential for our discourse.

\begin{defn}[Schauder Basis]\ \\
A \textit{Schauder basis} (also a \textit{countable basis}) of a (real or complex) Banach space $(X,\|\cdot\|)$ is a countably infinite set $\left\{e_n\right\}_{n\in \N}$ in $X$ such that
\begin{equation*}
\forall\, x\in X\ \exists! \left(c_k(x)\right)_{k\in\N}\in \F^\N:\ x=\sum_{k=1}^\infty c_k(x)e_k
\end{equation*}
($\F:=\R$ or $\F:=\C$), the series called the \textit{Schauder expansion} of $x$ 
and the numbers $c_k(x)\in \F$, $k\in\N$, the \textit{coordinates} of $x$ relative to $\left\{e_n\right\}_{n\in \N}$. 
\end{defn}

See, e.g., \cite{Markin2018EFA,Markin2020EOT,Lyust-Sob,SingerI}.

For an infinite-dimensional separable Hilbert space $(X,(\cdot,\cdot),\|\cdot\|)$
($(\cdot,\cdot)$ stands for \textit{inner product} and $\|\cdot\|$ for \textit{inner product norm}), an \textit{orthonormal basis} $\left\{e_n\right\}_{n\in \N}$ is a Schauder basis and for an arbitrary $x\in X$,
\begin{equation}\label{shs}
x=\sum_{k=1}^\infty c_k(x)e_k\quad \text{with}\quad c_k(x)=(x,e_k),\ k\in \N,
\end{equation}
(see, e.g., \cite{Markin2018EFA,Markin2020EOT}).

As we mention above, the sequence spaces $l_p$ ($1\le p<\infty$), $c_0$, and $c$ are examples of Banach spaces with a Schauder basis. For $l_p$ ($1\le p<\infty$) and $c_0$, the \textit{standard} such a basis is the set 
\[
\left\{e_n:=\left(\delta_{nk} \right)_{k\in \N}\right\}_{n\in \N}
\]
($\delta_{nk}$ is the \textit{Kronecker delta}) and, for an arbitrary $x:=(x_k)_{k\in\N}$ in the foregoing spaces,
\begin{equation*}
x=\sum_{k=1}^\infty c_k(x)e_k\quad \text{with}\quad c_k(x)=x_k,\ k\in \N,
\end{equation*}
(see, e.g., \cite{Markin2018EFA,Markin2020EOT,Lyust-Sob,SingerI}).

For the Banach space $c$ of \textit{convergent sequences} equipped with $\infty$-norm (see \eqref{infn}), the \textit{standard} Schauder basis is $\left\{e_n\right\}_{n\in \Z_+}$ ($\Z_+:=\left\{0,1,2,\dots\right\}$ is the set of \textit{nonnegative integers}) with
\[
e_0:=(1,1,1,\dots)
\]
and, for an arbitrary $x:=(x_k)_{k\in\N}\in c$,
\begin{equation}\label{c}
x=\sum_{k=0}^\infty c_k(x)e_k\quad \text{with}\quad c_0(x)=\lim_{m\to\infty}x_m,\ c_k(x)=x_k-c_0(x),\ k\in \N,
\end{equation}
(see, e.g., \cite{Markin2018EFA,Markin2020EOT,Lyust-Sob,SingerI}).

Banach spaces with more sophisticated Schauder bases encompass
$L_p(a,b)$ ($1\le p<\infty$) and $C[a,b]$ ($-\infty<a<b<\infty$) with $\infty$-norm
\[
C[a,b]\ni x\mapsto \|x\|_\infty:=\max_{a\le t\le b}|x(t)|
\]
(see, e.g., \cite{Lyust-Sob,SingerI}).

A Banach space with a Schauder basis is \textit{infinite-dimensional} and \textit{separable} (see, e.g., \cite{Lyust-Sob,Markin2018EFA,Markin2020EOT}). However, an infinite-dimensional separable Banach space need not have a Schauder basis (see \cite{Enflo1973}).

The set of $\F$-termed sequences
\begin{equation*}
Y:=\left\{y:=\left(c_k\right)_{k\in\N}\in \F^\N\,\middle|\, \sum_{k=1}^\infty c_ke_k\ \text{converges in}\ X\right\}
\end{equation*}
with termwise linear operations and the norm
\begin{equation*}
Y\ni y:=\left(c_k\right)_{k\in\N}\mapsto \|y\|_Y:=
\sup_{n\in\N}\left\|\sum_{k=1}^n c_ke_k\right\|
\end{equation*}
is a Banach space and the linear operator
\begin{equation*}
Y\ni y:=\left(c_k\right)_{k\in\N}\mapsto Ay:=
\sum_{k=1}^\infty c_ke_k\in X
\end{equation*}
is subject to the \textit{Inverse Mapping Theorem} (see, e.g., \cite{Dun-SchI,Lyust-Sob,Markin2018EFA,Markin2020EOT}). The \textit{boundedness}  of the inverse operator $A^{-1}:X\to Y$ implies \textit{boundedness}, and hence, \textit{continuity}, for the linear Schauder coordinate functionals 
\[
X\ni x=\sum_{k=1}^\infty c_k(x)e_k\mapsto c_n(x)\in \F,\ n\in \N,
\]
with
\[
\|c_n\|\le \dfrac{2\|A^{-1}\|}{\|e_n\|},\ n\in \N,
\]
(see, e.g., \cite{Lyust-Sob,Markin2018EFA,Markin2020EOT}) as well as for the linear operators
\begin{equation}\label{SnRn}
X\ni x=\sum_{k=1}^\infty c_k(x)e_k\mapsto S_nx:=\sum_{k=1}^n c_k(x)e_k,\
R_nx:=\sum_{k=n+1}^\infty c_k(x)e_k,\ n\in \N,
\end{equation}
with
\begin{equation}\label{idr}
I=S_n+R_n,\ n\in \N,
\end{equation}
($I$ is the \textit{identity operator} on $X$) and
\begin{equation}\label{SnRnN}
\|S_n\|\le \|A^{-1}\|\ \text{and}\ \|R_n\|\le 2\|A^{-1}\|,\ n\in \N,
\end{equation}
(see, e.g., \cite{Lyust-Sob}).

\begin{rem}
Here and henceforth, we use the notation $\|\cdot\|$ for the operator norm.
\end{rem}

\section{General Characterization}

The following statement appears to be a perfect illustration of the profound observation by Jacob T. Schwartz found in \cite{Schwartz1954} and chosen as the epigraph.

\begin{thm}[General Characterization of Convergence]\label{GCC}\ \\
Let  $(X,\|\cdot\|)$ be a (real or complex) Banach space  with a Schauder basis $\left\{e_n\right\}_{n\in \N}$ and corresponding coordinate functionals $c_n(\cdot)$, $n\in \N$. 

For a sequence $\left(x_n\right)_{n\in \N}$ and a vector $x$ in $X$,
\[
x_n\to x,\ n\to \infty,
\]
iff
\begin{enumerate}[label={(\arabic*)}]
\item\label{cond1} $\forall\, k\in\N:\ c_k\left(x_n\right)\to c_k(x)$, $n\to \infty$, and
\item\label{cond2} $\displaystyle \forall\, \varepsilon>0\ \exists\, K_0\in \N\ \forall\,K\ge K_0\ \forall\,n\in \N:\ \left\|R_Kx_n\right\|<\eps$.
\end{enumerate}
\end{thm}

\begin{proof}\

\textit{``Only if''} part. Suppose that, for a sequence $\left(x_n\right)_{n\in \N}$ and a vector $x$ in $X$,
\[
x_n\to x,\ n\to \infty.
\]

Then, by the continuity of the Schauder coordinate functionals $c_n(\cdot)$, $n\in \N$,  we infer that condition \ref{cond1} holds. 

Let $\eps > 0$ be arbitrary. Then, 
\begin{equation}\label{conv}
\exists\, N\in \N\ \forall\, n \geq N:\ \left\|x_n- x\right\|< \frac{\eps}{4\|A^{-1}\|}. 
\end{equation}

Since $x \in X$, 
\[
R_Kx:=\sum_{k=K+1}^\infty c_k(x)e_k\to 0,\ K\to \infty,
\]
and hence,
\begin{equation}\label{rem}
\exists\, K_0\in \N\ \forall\,K\ge K_0:\ \|R_Kx\|<\frac{\eps}{2}.
\end{equation}
		
In view \eqref{SnRnN}, \eqref{conv}, and \eqref{rem}, we have:
\begin{align*}
\forall\,K\ge K_0,\ \forall\, n\ge N:\ \left\|R_Kx_n\right\| & =\left\|R_Kx_n-R_Kx+R_Kx\right\|\\
&\leq \left\|R_K\left(x_n-x\right)\right\|+\left\|R_Kx\right\| \\
&\le \|R_K\|\left\|x_n-x\right\|+\|R_Kx\|\\
& < 2\|A^{-1}\|\frac{\eps}{4\|A^{-1}\|}+\frac{\eps}{2}=\frac{\eps}{2}+\frac{\eps}{2}= \eps.
\end{align*}

Further, since $x_n\in X$, $n=1,...,N-1$, we can regard $K_0\in \N$ in \eqref{rem} to be large enough so that
\[
\forall\,K\ge K_0,\ \forall\, n=1,...,N-1:\ \left\|R_Kx_n\right\|=\left\|\sum_{k=K+1}^\infty c_k\left(x_n\right)e_k\right\|<\eps.
\]

Thus, condition \ref{cond2} holds as well.

This completes the proof of the \textit{``only if''} part.

\smallskip
\textit{``If''} part. Suppose that, for a sequence $\left(x_n\right)_{n\in \N}$ and a vector $x$ in $X$, conditions \ref{cond1} and \ref{cond2} are met. 

For an arbitrary $\eps > 0$ and $K_0\in \N$ from condition \ref{cond2}, by condition \ref{cond1},
\begin{equation}\label{part}
\exists\, N \in \N\ \forall\, n \geq N:\
\left\|S_{K_0}\left(x_n-x\right)\right\| \leq \sum_{k=1}^{K_0} \left|c_k\left(x_n\right)-c_k\left(x\right)\right|\|e_k\|<\frac{\epsilon}{3}.
\end{equation}
		
Since $x \in X$, we can also regard that $K_0\in \N$ in condition \ref{cond2} to be large enough so that
\begin{equation}\label{rem2}
\|R_{K_0}x\|<\frac{\eps}{3}.
\end{equation}

Then, in veiw of \eqref{idr}, \eqref{part}, and \eqref{rem2} and by condition \ref{cond2}, 
\begin{align*}
\forall\, n \geq N:\ \|x_n-x\|&=\left\|S_{K_0}\left(x_n-x\right)+R_{K_0}\left(x_n-x\right)\right\|
\\
&\leq \left\|S_{K_0}\left(x_n-x\right)\right\|+\left\|R_{K_0}x_n\right\|+\|R_{K_0}x\|\\
&< \frac{\eps}{3}+\frac{\eps}{3}+\frac{\eps}{3} =\eps.
\end{align*}

This concludes the proof of the \textit{``if''} part and the entire statement.
\end{proof}

\begin{rems}\
\begin{itemize}
\item Condition \ref{cond1} is the convergence of the coordinates of $x_n$ 
to the corresponding coordinates of $x$ relative to $\left\{e_k\right\}_{k\in \N}$.
\item Condition \ref{cond2} signifies the uniform convergence of the Schauder expansions
\[
\sum_{k=1}^\infty c_k(x_n)e_k
\]
of $x_n$ relative to $\left\{e_k\right\}_{k\in \N}$ over $n\in \N$.
\end{itemize}
\end{rems}

Now, the \textit{Combined Characterization of Convergence} (Proposition \ref{CCC}) is an instant corollary of the foregoing general characterization.

\section{Characterization of Convergence in an Infinite-Dimensional Separable Hilbert Space}

For an infinite-dimensional separable Hilbert space $(X,(\cdot,\cdot),\|\cdot\|)$
relative to an \textit{orthonormal basis} $\left\{e_n\right\}_{n\in \N}$, in view of \eqref{shs}, the \textit{General Characterization of Convergence} (Theorem \ref{GCC}) acquires the following form.

\begin{cor}[Characterization of Convergence in a Separable Hilbert Space]\label{CCSHS}\ \\
Let  $(X,(\cdot,\cdot),\|\cdot\|)$ be a (real or complex) infinite-dimensional separable Hilbert space with an orthonormal basis $\left\{e_n\right\}_{n\in \N}$. 

For a sequence $\left(x_n\right)_{n\in \N}$ and a vector $x$ in $X$,
\[
x_n\to x,\ n\to \infty,
\]
iff
\begin{enumerate}[label={(\arabic*)}]
\item $\forall\, k\in\N:\ (x_n,e_k)\to (x,e_k)$, $n\to \infty$, and
\item $\displaystyle \forall\, \varepsilon>0\ \exists\, K\in \N\ \forall\,n\in \N:\ \sum_{k=K+1}^\infty {\left|(x_n,e_k)\right|}^2<\eps$.
\end{enumerate}
\end{cor}

\begin{rems}\
\begin{itemize}
\item Condition (1) is the convergence of the Fourier coefficients of $x_n$ to the corresponding Fourier coefficients of $x$ relative to $\left\{e_k\right\}_{k\in \N}$.
\item Condition (2) signifies the uniform convergence of the Fourier series expansions
\[
\sum_{k=1}^\infty (x_n,e_k)e_k
\]
of $x_n$ relative to $\left\{e_k\right\}_{k\in \N}$ over $n\in \N$.
\item The \textit{Characterization of Convergence in $l_p$} (Proposition \ref{CCl_p})
for $p=2$ is now a particular case of the prior characterization.
\end{itemize}
\end{rems}

\section{Characterization of Convergence in $c$}

Another immediate corollary of the \textit{General Characterization of Convergence} (Theorem \ref{GCC}) is the realization of the latter in the space $c$ of \textit{convergent sequences} equipped with $\infty$-norm (see \eqref{infn}) relative to the standard Schauder basis $\left\{e_n\right\}_{n\in \Z_+}$ (see Preliminaries).

Indeed, in $c$ relative to $\left\{e_n\right\}_{n\in \Z_+}$, for an arbitrary $x:=(x_k)_{k\in\N}$,
\[
x=\sum_{k=0}^\infty c_k(x)e_k\quad \text{with}\quad c_0(x)=\lim_{m\to\infty}x_m,\ c_k(x)=x_k-c_0(x),\ k\in \N,
\]
(see \eqref{c}) and
\begin{align*}
S_Kx&:=\sum_{k=0}^{K}c_k(x)e_k
=(\lim_{m\to \infty }x_m,x_1-\lim_{m\to \infty }x_m,\dots,
x_{K}-\lim_{m\to \infty }x_m,0,\dots),\\
R_Kx&:=\sum_{k=K+1}^\infty c_k(x)e_k
=(\hspace{-2mm}\underbrace{0,\dots,0}_{\text{$K+1$ terms}}\hspace{-2mm},x_{K+1}-\lim_{m\to \infty }x_m,\dots),\ K\in \Z_+,
\end{align*}
(cf. \eqref{SnRn}).

Thus, the \textit{General Characterization of Convergence} (Theorem \ref{GCC}),
in view of the obvious circumstance that, for any $x:=(x_k)_{k\in\N}\in c$, the sequence
\[
\|R_Kx\|=\sup_{k\ge K+1} \left|x_k^{(n)}-\lim_{m\to \infty }x_m^{(n)}\right|,\ K\in \Z_+,
\]
is decreasing, acquires the following form.

\begin{cor}[Characterization of Convergence in $c$]\label{CCc}\ \\
In the (real or complex) space $c$, 
\[
\left(x_k^{(n)}\right)_{k\in \N}=:x^{(n)}\to x:=\left(x_k\right)_{k\in \N},\ n\to \infty,
\]
iff
\begin{enumerate}[label={(\arabic*)}]
\item $\displaystyle \lim_{m\to \infty}x_m^{(n)}\to \lim_{m\to \infty}x_m$, $n\to \infty$,
\quad  and \quad 
$\displaystyle \forall\, k\in\N:\ x_k^{(n)}\to x_k$, $n\to \infty$;
\item $\displaystyle \forall\, \varepsilon>0\ \exists\, K\in \Z_+\ \forall\,n\in \N:\ \sup_{k\ge K+1} \left|x_k^{(n)}-\lim_{m\to \infty }x_m^{(n)}\right|<\varepsilon$.
\end{enumerate}
\end{cor}

\begin{rems}\
\begin{itemize}
\item Condition (1), beyond termwise convergence, includes convergence of the limits.
\item Condition (2) signifies the uniform convergence of the sequnces $\left(x_k^{(n)}\right)_{k\in \N}$ to their respective limits over $n\in \N$.
\item The \textit{Characterization of Convergence in $c_0$} (Proposition \ref{CCc_0})
is a mere restriction of the prior characterization to the subspace $c_0$ of $c$.
\end{itemize}
\end{rems}

\section{Concluding Remark}

As is easily seen, the \textit{General Characterization of Convergence} (Theorem \ref{GCC}) is consistent with the following \textit{characterization of compactness}, which underlies the results of \cite{Markin1990}.

\begin{thm}[Characterization of Compactness {\cite[Theorem III.$7.4$]{Lyust-Sob}}]\label{CP}\ \\
In a (real or complex) Banach space $(X,\|\cdot\|)$ with a Schauder basis, a set $C$ is precompact (a closed set $C$ is compact) iff
\begin{enumerate}[label={(\arabic*)}]
\item $C$ is bounded and
\item $\displaystyle \forall\, \varepsilon>0\ \exists\, K_0\in \N\ \forall\,K\ge K_0\ \forall\,x\in C:\ \left\|R_Kx\right\|<\eps$.
\end{enumerate}
\end{thm}

 
\end{document}